\documentclass[12pt]{amsart}

\usepackage{fullpage, color}

\usepackage{amssymb, amsmath, amsthm}
\usepackage{hyperref}
\newcommand{\uhp}{\mathbb{H}}
\newcommand{\R}{\mathbb{R}}
\newcommand{\C}{\mathbb{C}}
\newcommand{\D}{\mathbb{D}}
\newcommand{\McC}{\raise.5ex\hbox{c}}
\newcommand{\T}{\mathbb{T}}

\newtheorem{theorem}{Theorem}[section]
\newtheorem{prop}[theorem]{Proposition}

\newtheorem{lemma}[theorem]{Lemma}

\theoremstyle{remark}
\newtheorem{example}[theorem]{Example}
\newtheorem{claim}[theorem]{Claim}
\newtheorem{remark}[theorem]{Remark}

\title{Stable polynomials and admissible numerators in product domains}
\author[Bickel]{Kelly Bickel}
\address{Department of Mathematics, Bucknell University, Lewisburg, PA 17837,  USA}
\email{kelly.bickel@bucknell.edu}

\author[Knese]{Greg Knese} 
\address{Department of Mathematics, Washington University in St Louis, St Louis, MO 63130, USA}
\email{geknese@wustl.edu}

\author{James Eldred Pascoe} 
\address{Department of Mathematics, Drexel University, Philadelphia, PA 19104, USA}
\email{jep362@drexel.edu}

\author{Alan Sola}
\address{Department of Mathematics, Stockholm University, 10691 Stockholm, Sweden}
\email{sola@math.su.se}
\date{\today}

\thanks{KB partially supported by NSF grant DMS-2000088.\\ GK partially supported by NSF grant DMS-2247702.\\
JEP  partially supported by NSF grant DMS-2319010
\\AS acknowledges support from Ivar Bendixsons stipendiefond f\"or docenter.
} 
\dedicatory{Dedicated to John E. M\McC Carthy on the occasion of his sixtieth birthday}

\begin{document}

\begin{abstract}
Given a polynomial $p$ with no zeros in the polydisk, or equivalently the poly-upper half-plane, 
we study the problem of determining the ideal of polynomials $q$ 
with the property that the rational function $q/p$ is bounded near a boundary zero of $p$. 
We give a complete description of this ideal of numerators 
in the case where the zero set of $p$ is smooth 
and satisfies a non-degeneracy condition.
%positivity condition. 
We also give a description of the 
ideal in terms of 
an integral closure when $p$ has an isolated zero on the distinguished
boundary.
 Constructions of multivariate stable 
polynomials are presented to illustrate sharpness of our results 
and necessity of our assumptions.
\end{abstract}
\maketitle

\section{Introduction}
Let $d\geq 1$ and $\Omega \subset \mathbb{C}^d$ be a domain. A polynomial $p\in \mathbb{C}[z_1,\ldots, z_d]$ is said to be {\it stable} 
with respect to $\Omega$ if $p(z)\neq 0$ for $z\in \Omega$. Stable polynomials have many important applications (see, for instance, the survey \cite{wag} and the introduction in \cite{BKPS} and the references provided there); among others, they serve as denominators of rational functions that are holomorphic in $\Omega$. If $p$ is in fact strictly stable, meaning that $p(z) \neq 0$ for $z\in \overline{\Omega}$, then $q/p$ is automatically analytic in $\Omega$ and smooth on its closure. However, if $p$ is stable but has zeros on $\partial \Omega$, then one is immediately faced with the interesting problem of finding conditions on $q\in \mathbb{C}[z_1,\ldots, z_d]$ that guarantee that $q/p$ has good properties in $\Omega$ in addition to merely being analytic.

In this article we study what we call the {\it admissible numerator problem} on one of the standard reference domains in $\mathbb{C}^d$, the unit polydisk 
\[\D^d =\{(z_1,\dots, z_d) \in \C^d: |z_1|<1,\dots, |z_d|<1\}.\]
Given a stable $p$, our task is to determine when a rational function $q/p$
in three or more variables is bounded on $\mathbb{D}^d$. This is part of a more general program to
understand singularities of rational
functions on a boundary of a domain and what constraints
they force on the function.  The papers \cite{BKPS, Kollar} addressed the two-variable admissible numerator problem on
the bidisk, and together gave a full characterization of admissible numerators
$q\in \mathbb{C}[z_1,z_2]$ associated to a polynomial stable on the bidisk.  
This was possible because of a complete
local description of the zero set of a stable polynomial 
in $\mathbb{D}^2$ with zeros on the boundary (see \cite[Section 2]{BKPS} for details).  
In dimensions three or higher, such a local 
description is lacking, and, as we will see, the admissible numerator problem becomes more complicated. For a two-variable $L^{\mathfrak{p}}$ version of the problem we consider here, see \cite{Kprep} which also contains an alternative proof of the characterization of the ideal of admissible numerators.

\begin{remark}
A rational function $q/p$, analytic on $\D^{d}$, 
is bounded on $\D^{d}$ if and only if it is essentially bounded
on $\T^{d}$.  This follows, for instance, from classical
Hardy space theory on the polydisk since $p$ is outer.
A broad strategy for approaching the admissible numerator
problem is then to understand when $q/p$ is bounded
in a neighborhood within $\T^{d}$ 
of each point $w\in Z_{p}\cap \T^{d}$;
equivalently, we can try to understand when $q/p$ is bounded
on a neighborhood of $w \in Z_{p}\cap \T^{d}$ within $\C^{d}$ intersected with 
$\D^{d}$. By compactness of $\T^{d}$, if $q/p$ is essentially
bounded in a neighborhood of every point of $\T^{d}$, then
$q/p$ will be essentially bounded on all of $\T^{d}$.
\end{remark}

Thus, as in \cite{BKPS} we can examine a local problem and we perform conformal
maps in order to transfer to the biholomorphically equivalent situation of
the poly-upper half-plane 
\[
\uhp^{d+1} =\{(x_1,\dots, x_d,z): \Im x_1,\dots, \Im x_d, \Im z>0\},
\]
where we have singled out a distinguished variable $z$. When we restrict to the three variable setting, we will use the notation $(x,y,z)$ and the distinguished variable will still be $z$.
In the general setting, given $p,q \in \C[x_1,\dots, x_d,z]$
where $p(x_1,\dots, x_d, z) \ne 0$ for
$\Im x_j >0, \Im z>0$ and $p(0)=0$,
we wish to know when $q(x,z)/p(x,z)$
is bounded on a neighborhood
of $0\in \C^{d+1}$ intersected with $\uhp^{d+1}$. 
Working in the specific setting of a product domain gives us 
several advantages as regards the admissible numerator problem. 
Given a stable $p$, it is easy to exhibit at least one non-trivial 
admissible numerator in the form of the reflection polynomial 
associated with $p$, given in the poly-upper half-plane case by
$\bar{p}(w)=\overline{p(\bar{w})}$, 
and so we immediately get existence of admissible numerators 
other than $p$ when $p$ is not a multiple of $\bar{p}$. 
In addition, there are several methods for constructing 
stable polynomials in the polydisk or poly-upper half-plane 
with prescribed properties (see for instance \cite{BPS2, pascoe18,S23}, 
and results presented below). 
As we will see, however, not all admissible numerators can 
generally be produced from $p$ and $\bar{p}$, or indeed just two fixed polynomials.

We have singled out a distinguished variable $z$
because in this article we investigate the
simplest type of boundary singularity,
namely,
when $\frac{\partial p}{\partial z}(0) \ne 0$
and the zero set $\mathcal{Z}_p=\{(x,z): p(x,z) = 0\}$
is parametrized via
an analytic function by the implicit function theorem.
Geometrically we are assuming the zero set
of $p$ is a smooth variety through $0$.
In particular, by the Weierstrass 
preparation theorem we can factor
\begin{equation}
p(x,z) = u(x,z)\left(z + \phi(x)\right)
\label{pfac}
\end{equation}
for $\phi(x) \in\C\{x_1,\dots, x_d\}$ with $\phi(0)=0$, and $u(x,z) \in \C\{x_1,\dots,x_d,z\}$ having $u(0,0)\ne 0$.
Here $\C\{x_1,\dots, x_d\}$ and related variations denote rings of
power series convergent in a neighborhood of $0$ in the given variables.

The first step is to give a description of the relevant $\phi$.

\begin{prop}\label{phithm}
Suppose $\phi(x_1,\dots,x_d) \in \C\{x_1,\dots,x_d\}$; i.e.\ is analytic near $0$.  
Assume $\phi(0)=0$, 
and $z + \phi(x)$ has no zeros for $z \in \uhp$ and (small) $x \in \uhp^d$.
Then, $\nabla \phi(0) \in [0,\infty)^d$ and $\Im \phi$ is
non-negative on a neighborhood of $0$ within $\R^{d}$.
In addition, either:
\begin{itemize}
\item $\phi \in \R\{x_1,\dots,x_d\}$; i.e.\ it has all
real coefficients or 
\item $\phi$ has the form
\[
\phi(x) = \phi_1(x) + \cdots +\phi_{2L-1}(x) + \phi_{2L}(x)+\cdots
\]
where each $\phi_j$ is a homogeneous polynomial of degree $j$,
$\phi_1,\dots, \phi_{2L-1} \in \R[x_1,\dots, x_d]$ have real coefficients,
and $\Im \phi_{2L} \not\equiv 0 $ is non-negative on $\R^d$.
\end{itemize}
Furthermore, 
if  $T \subset \{1,\dots, d\}$ is the set of components of $\nabla \phi(0)$ that are zero, then $\phi(\sum_{j\in T} x_j e_j) \equiv 0$. (Here $e_j$ are standard basis vectors for $\C^d$.)
\end{prop}

When we write $\Im \phi$ we are referring
to the analytic function obtained by extracting the
imaginary parts of the coefficients of $\phi$,
namely
\[
(\Im \phi)(x) := \frac{1}{2i}(\phi(x) - \overline{\phi(\bar{x}}))
\]
which does equal $\Im (\phi(x))$ when $x \in \R^d$.
Note that $\nabla \phi(0) = 0$ implies $\phi\equiv 0$.
Note also, if $\phi$ has an isolated zero at $0\in \R^d$ then
necessarily $\nabla \phi(0) \in (0,\infty)^d$.

From here we wish to describe the following {\it ideal of admissible numerators}
\[
\mathcal{I}^{\infty}_p = \{ q(x,z) \in \C\{x_1,\dots, x_d,z\}: 
q/p \text{ is bounded on } \uhp^{d+1}\cap \D_{\epsilon}^{d+1}
\text{ for some } \epsilon>0\}
\]
in as simple terms as possible.
Here $\D_{\epsilon} = \{z \in \C: |z|<\epsilon\}$. 
Because we are working in a small neighborhood of the origin,
most all inequalities to follow will only be proven/stated for 
inputs sufficiently close to the origin.
We shall use the common notations 
\[
f(x) \gtrsim g(x),\quad  f(x) \lesssim g(x), \quad f(x) \asymp g(x)
\]
to denote
inequalities of the form 
\[
f(x) \geq C g(x), \quad f(x) \leq C g(x), \quad c f(x) \leq g(x) \leq Cg(x)
\]
near $0$ for positive constants $c,C>0$ whose values are of no importance.
\begin{remark}\label{realremark}
When $\phi \in \R\{x_1,\dots, x_d\}$
in Proposition \ref{phithm} then the ideal $\mathcal{I}_p^{\infty}$
is the principal ideal $(p)$ since all numerators will vanish on the smooth
variety $z+\phi(x)=0$. 
This follows from the Weierstrass division theorem.
\end{remark}

Outside of the case of Remark \ref{realremark}, 
the simplest polynomials to consider are those where $\Im \phi_{2L}(x)$ is positive definite on $\R^d$
in the sense that $\Im \phi_{2L}(x)$ is strictly positive on $\{x\in \R^d:|x|=1\}$
or equivalently $\Im \phi_{2L}(x) \asymp |x|^{2L}$. 
In this case, the zero set of $z+\phi(x)$ on $\R^d$ is automatically isolated at $0$. 

\begin{theorem} \label{thmImpos}
Assume $p(x_1,\dots, x_d,z) \in \C[x_1,\dots, x_d,z]$ 
has no zeros in $\uhp^{d+1}$, $p(0)=0$, 
and $\frac{\partial p}{\partial z}(0)\ne 0$.
Parametrize the zero set of $p$ near $0$ as in 
\eqref{pfac} and
Proposition \ref{phithm}, i.e. as $z+\phi(x) = 0$.
If $\phi(x) \notin \R\{x_1,\dots, x_d\}$
and the first homogeneous term $\phi_{2L}$
with non-trivial imaginary part
satisfies the condition that $\Im \phi_{2L}$ is positive definite,
then the ideal of admissible numerators for $p(x,z)$ is given by
\[
\mathcal{I}^{\infty}_{p} = (z + q(x), (x)^{2L})
\]
where $q(x) = \sum_{j<2L} \phi_{j}(x)$ and $(x)^{2L}$ is the ideal generated by powers $x^{\alpha}$ for $|\alpha| = 2L$.
\end{theorem}
This result already covers many natural examples such as the $\mathbb{D}^d$-stable polynomial 
$d-\sum_{j=1}^{d}z_j$. 
We present the details for $d=3$, using the variables $(x,y,z)$ to lighten notation.

\begin{example}\label{faveex}
Consider the tridisk example $3-z_1-z_2-z_3$ which converts
to the tri-upper-half-plane stable polynomial
\[
\begin{aligned}
p(x,y,z) &= x+y+z - 2i (xy+xz+yz) - 3x yz \\
&= (1-2i(x+y) - 3x y)\left(z + \underset{=: \phi(x,y)}{\underbrace{\frac{x+y-2i x y}{1-2i(x+y) -3x y}}}\right)\\
%&= (1-2i(x+y) - 3x y)(z + \phi(x ,y))
\end{aligned}
\]
where
\[
%\begin{aligned}
\phi(x ,y) %&= x +y + (-2ix y + 2i(x +y)^2) + \text{ higher order}\\
%&
= x +y + 2i(x^2+x y+y^2) + \text{ higher order terms}.
%\end{aligned}
\]
We have $\Im \phi_2(x,y) = 2(x^2+xy+y^2)$ which is positive 
definite.
Therefore, the set of functions $q(x ,y,z)$ analytic at $(0,0,0)$ such that 
\[
\frac{q(x,y,z)}{p(x,y,z)}
\]
is locally bounded near $(0,0,0)$ in $\R^3$ is exactly the ideal
\[
(x+y+z, x^2, x y, y^2)
\]
within the ring of convergent power series.  This is simply the functions with
first order term given by a multiple of $x+y+z$. 

Converting back to $\D^3$ and the polynomial
$3-z_1-z_2-z_3$ we have that the ideal of numerators
$q(z_1,z_2,z_3)$ such that $q(z_1,z_2,z_3)/(3-z_1-z_2-z_3)$ is bounded on $\D^3$ is generated by
\[
(3-z_1-z_2-z_3, (1-z_1)^2, (1-z_1)(1-z_2), (1-z_2)^2). \qquad \diamond
\]
\end{example}

The description of $\mathcal{I}^{\infty}_p$ in Theorem \ref{thmImpos} is reminiscent of the case of order one vanishing in two variables, see \cite[Theorem 5.4]{BKPS}, where the corresponding ideal is generated by 
a pair of elements of the form $z+q(x_1)$ and $x_1^{2L}$, $L\in \mathbb{N}$ with $L\geq 1$.
Before we move on to a discussion of whether this similarity persists in the absence of the positivity condition imposed in Theorem \ref{thmImpos}, we note that while Example \ref{faveex} exhibits the simplest situation $L=1$ it is of interest to know if higher order vanishing of $\Im \phi$
is possible.  Namely, does the global constraint of $p$ having no zeros on $\uhp^d$ force any special behavior
(as happens with $\nabla \phi(0)$)? The next result discussed below shows that any degree $2L \in \mathbb{N}$ can occur in Theorem \ref{thmImpos}. 

\iffalse
First, it is worth noting that the rest of the material we wish to present is formulated in the three-variable case. 
This is for notational convenience when constructing and examining examples below, 
but in the case of the final theorem presented below,  
our proof is restricted to dimension three.  
(That said, we do have a general result below, Proposition \ref{genprop}, that
leaves room for generalizing to $d$ variables.)
\fi
Again in three variables, it is convenient to simply use $(x,y,z)$ as variables with 
$z$ still acting as our distinguished variable. 
\begin{prop}\label{prop:sqpowers}
For any $L\in \mathbb{N}$ there exists a polynomial $p(x,y,z) \in \C[x,y,z]$ 
with no zeros in $\uhp^3$, 
its sole zero in $\overline{\mathbb{H}^3}$ at $(0,0,0)$, $\frac{\partial p}{\partial z}(0) \ne 0$, 
and the associated $\phi$ having $\Im \phi$ vanishing to order $2L$ with $\Im \phi_{2L}$ positive definite.
\end{prop}

The next natural questions are if it is possible for $\Im \phi_{2L}$ to
fail to be positive definite for an isolated zero, and what happens in this case.

We have a construction for producing a family of examples
with degenerate $\Im \phi_{2L}$.
The construction is easier in the setting of the tridisk
where we start with a polynomial $p_0(x,z)$ with no
zeros in $\D^2$ but a zero at $(1,1)$ and consider
\[
p_1(x,y,z) = p_0((x+y)/2,z).
\]
The upper half-plane version of this simply looks more complicated.
First, let $\beta = i(1-z)/(1+z)$ be a standard conformal map from $\mathbb{D}$ to $\mathbb{H}$ and note that
\[
\beta\left( \frac{\beta^{-1}(x) + \beta^{-1}(y)}{2}\right) = \frac{i(x+y) +2xy}{2i+x+y}.
\]
Then, we have the following. We make
reference to a property of a two-variable stable polynomial
called \emph{contact order} that we define later;
it essentially measures the rate that the polynomial's 
zero set approaches $\R^2$.

\begin{prop} \label{prop:degex} Let $q(x,y) \in \C[x,y]$ have no zeros in $\uhp^2$,
bidegree $(m,n)$, finitely many zeros on $\R^2$, and $q(0,0)=0$.
Assume 
$\frac{\partial q}{\partial  y}(0) \ne 0,$ and that $q$ has 
contact order $K>2$ at $(0,0)$. Set
\[ 
p(x,y,z) =
(2i + x+y)^m q\left( \frac{i(x+y) +2xy}{2i+x+y}, z\right).\]
Then, $p(x,y,z)\ne 0$ for $(x,y,z) \in \mathbb{H}^3$ and $p$ has 
finitely many zeros on $\mathbb{R}^3$.
Furthermore, there is a parametrization of $\mathcal{Z}_p$ near $0$ of the form $z+\phi(x,y)=0$ such that  $\Im \phi_2(x,y) \not \equiv 0$ and $\Im \phi_k(x,y) =0$ when $x=y$ for $k <K.$
\end{prop}

There are known constructions that give polynomials $q$ as described in Proposition \ref{prop:degex}. For example, at the end of \cite{pascoe18}, Pascoe constructs a family of rational functions of degree $(n,1)$ with prescribed regularity at a boundary point. As described in Section 3.4 in \cite{BKPS}, this boundary regularity condition actually translates to a statement about contact order and so, Pascoe's construction yields 
stable polynomials of degree $(n,1)$ with prescribed contact order $K$ at $(0,0)$ and finitely many zeros in $\R^2$. Combined with Proposition \ref{prop:degex}, this  gives us a degenerate example for each possible $K.$ 
Another construction of general $K$-contact order polynomials with a prescribed number
of zeros is given by Sola in \cite[Corollary 6]{S23}.

In the general case of degenerate $\Im \phi_{2L}$ but also when $p$ has
an isolated zero on $\R^d$ it is possible to reduce the description of $\mathcal{I}^{\infty}_{p}$
to an ideal described entirely with polynomials.
The first step is showing that if $\Im \phi(x) >0$ near $0\in \R^d$ but $\phi(0)=0$,
then $\Im \phi$ is comparable to a polynomial.

\begin{lemma} \label{lociso}
Suppose $f \in \R\{x_1,\dots, x_d\}$, $f(0)=0$, and $f(x) >0$ for small $x\ne 0$.
Then, there exist a positive integer $K$ and a polynomial 
$g(x_1,\dots, x_d) \in \R[x_1,\dots, x_d]$ such that near $0$
\[
g(x) \asymp f(x) 
\]
and $f(x) \gtrsim |x|^{K}$.
\end{lemma}

\begin{remark} \label{kollarremark}
We are grateful to J. Koll\'ar for pointing out that
this result follows directly from the Łojasiewicz inequality.  
Our original draft stated and proved the above lemma for $d=2$.
We shall include our original explanation and proof for $d=2$
since it shows how to construct $g$ and $K$.   
\end{remark}

Once we have this lemma in place,
we can prove the following.

\begin{theorem}\label{thm:isoideal}
Let $p(x_1,\dots, x_d,z)\in \C[x_1,\dots, x_d,z]$ have no zeros in $\uhp^{d+1}$, 
$p(0) = 0$, $\frac{\partial p}{\partial z}(0) \ne 0$,
and assume the zero $0$ of $p$ is isolated with respect to $\R^{d+1}$.

Then, there exist $g(x) \in \R[x_1,\dots, x_d]$ and $H(x) \in \R[x_1,\dots, x_d]$
such that the ideal of admissible numerators for $p(x,z)$
is given by
\[
\mathcal{I}^{\infty}_p = ( z+H(x), \mathrm{IC}(g(x)) )
\]
where
\[
\mathrm{IC}(g(x)) = \{ q(x) \in \C\{x_1,\dots, x_d\}:  |q(x)| \lesssim g(x) \text{ near } (0,0)\}.
\]

\end{theorem}
This transfers the problem of determining
the full ideal of admissible numerators
to the real algebraic geometry problem of determining
$\mathrm{IC}(g(x))$. The notation here stands for {\it integral closure}, see \cite{Kollar} for details and references.

The remainder of our paper is structured as follows. 
In Section \ref{exsection} we give an example of Theorem \ref{thm:isoideal}
in action to exactly determine an admissible numerator ideal
in the absence of the positive definite condition in Proposition \ref{phithm}. 
Section \ref{proofspI} contains proofs of the first four of our stated results. 
In Section \ref{proofspII}, we present the proofs of Lemma \ref{lociso} and Theorem \ref{thm:isoideal}
Along the way, we examine several examples that illustrate the results obtained in this paper.
%%%%%%%%%%%%%%%%%%%%%%%%%%%%%%%%%%%%%%%%%%%%%%

\section{A more complicated example}\label{exsection}
Before we turn to the proofs of our results, we examine a more involved example. The goal is to get
a feeling for what Theorem \ref{thm:isoideal} says in practice, and to see what 
computations look like when we are faced with a non-definite first imaginary homogeneous term at an isolated zero at $0\in \R^3$.
\begin{example}\label{nondefex}
We start with the polynomial
%with no zeros in $\D^2$ 
given by 
\[
p_0(x,y) = x^{2} - x y - 3 \, x - y + 4.
\]
Note that $p_0(x,y) = (2-x-y)( 2- x - \frac{2xy-x-y}{2-x-y})$, so 
we see that $p_0$ has no zeros in $\D^2$.
Then, consider 
\[
p_1(x,y,z) = p_0((x+y)/2,z) = \frac{1}{4} \, x^{2} + \frac{1}{2} \, x y + \frac{1}{4} \, y^{2} - \frac{1}{2} \, x z - \frac{1}{2} \, y z - \frac{3}{2} \, x - \frac{3}{2} \, y - z + 4.
\]
If we convert this to a polynomial with no zeros on $\uhp^3$ but with a zero at $0$ we get
\[
\begin{aligned}
p(x,y,z) &= 2 \, x^{2} y^{2} z + 2 i \, x^{2} y^{2} + 3 i \, x^{2} y z \\
&+ 3 i \, x y^{2} z - \frac{5}{2} \, x^{2} y - \frac{5}{2} \, x y^{2} - \frac{5}{4} \, x^{2} z \\&
- \frac{9}{2} \, x y z - \frac{5}{4} \, y^{2} z - \frac{3}{4} i \, x^{2} - \frac{5}{2} i \, x y \\
&- \frac{3}{4} i \, y^{2} - 2 i \, x z - 2 i \, y z + \frac{1}{2} \, x + \frac{1}{2} \, y + z.
\end{aligned}
\]
We can directly solve for a parametrization of the zero set in the form $z+\phi(x,y)$ 
and see that $\phi$ has the initial power series expansion
\[
\begin{aligned}   
\phi(x,y) &= \frac{1}{2} ( x +  y)\\
&+\frac{i}{4}(x-y)^2 \\
%\frac{1}{4} i \, x^{2} - \frac{1}{2} i \, x y + \frac{1}{4} i \, y^{2}\\
&+\frac{1}{8} ( x^{3} + 7 \, x^{2} y + 7\, x y^{2} + y^{3})\\
&+\frac{1}{16} i \, {\left(9 \, x^{2} - 2 \, x y + 9 \, y^{2}\right)} {\left(x + y\right)}^{2} + \text{ higher order terms.}
%\frac{9}{16} i \, x^{4} + i \, x^{3} y + \frac{7}{8} i \, x^{2} y^{2} + i \, x y^{3} + \frac{9}{16} i \, y^{4} + \text{ higher order}
\end{aligned}
\]
Here $\Im \phi_2(x,y) = \frac{1}{4}(x-y)^2$ is evidently not positive definite.
This is compensated by the 
term $\Im \phi_4(x,y) = \frac{1}{16} \, {\left(9 \, x^{2} - 2 \, x y + 9 \, y^{2}\right)} {\left(x + y\right)}^{2}$,
which though it is not positive definite either, the sum of the terms will be positive
except at $0$.  

One can show 
\[
\Im \phi(x,y) \asymp (x-y)^2 + (x^2+y^2)(x+y)^2
\]
for $(x,y) \in \R^2$ close to $0$. 
Indeed, this follows from
\[
\Im \phi(x,y) = \frac{1}{4}(x-y)^2 + 
\frac{1}{16} \, {\left(9 \, x^{2} - 2 \, x y + 9 \, y^{2}\right)} {\left(x + y\right)}^{2}
+ O(|(x,y)|^5)
\]
as well as 
\[
\left(9 \, x^{2} - 2 \, x y + 9 \, y^{2}\right) \asymp x^2 +y^2
\]
and
\[
|(x,y)|^4 \lesssim (x-y)^2 + (x^2+y^2)(x+y)^2.
\]

Next, let $g(x,y) = (x-y)^2 + (x^2+y^2)(x+y)^2$.  
\begin{claim}
The ideal 
$\mathrm{IC}(g(x,y))$ of
$q(x,y) \in \C\{x,y\}$
satisfying 
$|q(x,y)| \lesssim g(x,y)$
for $x,y$ small is given by
\[
\mathrm{IC}(g(x,y)) = ((x-y)^2, (x-y)(x+y)^2, (x,y)^4).
\]
\end{claim}

\begin{proof}
We shall change variables to  $u=x-y, v=x+y$ and consider $G(u,v) = u^2+(u^2+v^2)v^2 = u^2 + u^2 v^2+v^4$. 
The ideal of polynomials satisfying $|q(u,v)| \lesssim G(u,v)$ evidently contains 
\[
u^2, u^2v^2, uv^3, v^4.
\]
It also includes $uv^2$ because 
\[
|u|v^2 \leq \frac{1}{2}(u^2 + v^4).
\]
The ideal contains all fourth degree monomials.  The only third degree monomial
that it does not contain is $v^3$.  The only second degree monomial it does contain is $u^2$.  Indeed,
we cannot have
\[
|uv| \lesssim (u^2+u^2v^2+v^4)
\]
because if we set $u=v^2$ we get $v^3\lesssim v^4$.  Also, $v^2$ is ruled out by setting $u=0$.

The ideal therefore contains
\[
I_0 = (u^2, u^3, u^2v, uv^2, (u,v)^4).
\]
(Elements $u^3,u^2v$ are written for emphasis.)
To prove this is all, we take a polynomial $q(u,v)$ bounded above by $u^2+u^2v^2+v^4$
and reduce modulo $I_0$.  We can write 
\[
q(u,v) = a + bu +cv + d uv + ev^2 + g v^3.
\]
We assume $|q(u,v)| \lesssim (u^2+u^2v^2+v^4)$ for $u,v$ small.
Evidently, $a=0$.  Setting $v=0$ yields $b=0$. Setting $u=0$ yields $c=0$ then $e=g=0$.  
Finally, $uv$ is not in the ideal so $d=0$. 
\end{proof}
Now we can apply Theorem \ref{thm:isoideal} to write down the ideal of admissible numerators of $p$. 
Using the notation of Lemma \ref{lociso}, we have $K=4$ for $\Im \phi$. 
Then the proof of Theorem \ref{thm:isoideal} implies that $H$ is the 
third-order Taylor polynomial of $\Re \phi$. Specifically,
\[
H(x,y) = \frac{1}{2} ( x +  y)  
+\frac{1}{8} ( x^{3} + 7 \, x^{2} y + 7\, x y^{2} + y^{3}).
\]
Theorem \ref{thm:isoideal} immediately implies that the ideal  is
\[
\mathcal{I}^{\infty}_{p} = 
(z+H(x,y), (x-y)^2, (x-y)(x+y)^2, (x,y)^4). \quad \diamond
\]
\end{example}

\section{Proofs, part I: general dimensions and example constructions}\label{proofspI}

The following is a straightforward exercise using a local
description of analytic functions in one variable so
we omit the proof.

\begin{lemma} \label{onevarlemma}
Suppose $f(x) \in \C\{x\}$ maps elements of $\uhp$ to $\overline{\uhp}$
and $f(0)=0$. If $f\not\equiv 0$ then $f'(0) > 0$.
\end{lemma}

\begin{proof}[Proof of Proposition \ref{phithm}]
We may assume $\phi \not\equiv 0$.
Locally, $\phi$ maps $\uhp^d$ to $\overline{\uhp}$;
so evidently $\Im \phi \geq 0$ on $\R^d$.
For any $v \in [0,\infty)^d\setminus \{0\}$, $t \mapsto \phi(tv)$ maps $\uhp$ to $\overline{\uhp}$ (locally).
By Lemma \ref{onevarlemma}, either this map is identically zero or 
$\nabla \phi(0) \cdot v >0$.   If $\nabla \phi(0) = 0$, then we would have $\phi$ equal
to zero on an open subset of $\R^d$ implying $\phi \equiv 0$.
Thus, $\nabla \phi(0) \in [0,\infty)^{d} \setminus \{0\}$.

Suppose $\phi$ does not have all real coefficients.
Then, we may write a homogeneous expansion where $\phi_M$ is
the first term with complex and non-real coefficients
\[
\phi(x) = \phi_1(x) + \cdots +\phi_{M-1}(x) + \phi_{M}(x) + \cdots.
\]
Since $\Im \phi(x) \geq 0$ for small $x \in \R^d$,
 \[
 \lim_{t\searrow 0} \Im \phi(tx)/t^M = \Im \phi_M(x) \geq 0.
 \]
We necessarily have that $M$ is even for otherwise $\Im \phi_M(-x) = -\Im \phi_M(x) \geq 0$
would imply $\Im \phi_M$ is identically zero.  

The final claim follows from restricting $\phi$ to the components corresponding to zero components of
$\nabla \phi(0)$ and repeating the argument just given.
\end{proof}

%Also, note that for $d=2$, if $\phi(x_1,0) \equiv 0$ then $x_2$
% divides $\phi$ and the line $t\mapsto (t,0,0)$ lies in the zero set of $z+\phi(x)$.  

\begin{remark} \label{phiremark}
Assuming the setup of Proposition \ref{phithm} with $\phi \notin \R\{x_1,\dots, x_d\}$, 
$z+\phi(x)$ has an isolated zero at $0$ with respect to $\R^{d+1}$
if and only if 
$\Im \phi(x)$ has an isolated zero at $0$ with respect to $\R^{d}$.
Indeed $z + \phi(x) = 0$ if and only if $z = -\Re \phi(x)$ and $\Im \phi(x) = 0$.

Note that if $\nabla \phi(0)$ has a zero component, 
say the first, then $\phi(x_1,0,\dots, 0)\equiv 0$ and 
$\Im \phi (x)$ does not have an isolated zero.
So we obtain the conclusion that
if $\Im \phi(x)$ has an isolated zero then $\nabla \phi(0)$
has all positive entries; i.e. the imaginary part is
putting constraints on the real part of $\phi$.

Finally, note that if 
$\Im \phi_{2L} (x)$ is positive definite then
$\Im \phi(x) = \Im \phi_{2L}(x) + O(|x|^{2L+1}) \asymp |x|^{2L}$
has an isolated zero
which by the above implies $\nabla \phi(0)$ has all positive
entries. $\quad \diamond$
\end{remark}

The condition that none of the components of $\nabla \phi(0)$ 
 vanishes gives us local control over $\Im(\phi(x))$ in $\uhp^d$.
 
\begin{prop} \label{Im-est}
   Suppose $\phi(x) \in \C\{x_1,\dots, x_d\}$, $\phi(0) = 0$, 
and $\Im \phi(x) \geq 0$ for $x \in \R^d$.    
Assume $\nabla \phi(0) \in (0,\infty)^{d}$.
 
    Then, for $x = u+iv \in \overline{\uhp}^d$ close to $0$ we have
    \[
    \Im (\phi(x)) \gtrsim ((\Im \phi (u)) + |v|) \gtrsim (|(\Im \phi)(x)| + |v|),
    \]
    where $(\Im \phi)(x) = \frac{1}{2i} (\phi(x) - \overline{\phi(\bar{x})})$
    while $\Im (\phi(x))$ is the pointwise imaginary part of $\phi(x)$.
\end{prop}

\begin{proof}
Let us write $\phi(x) = A(x) + i B(x)$ where $A = \Re\phi$ and $B = \Im\phi$.

We claim first that
\[
\Im(A(u+iv)) = \nabla \phi(0) \cdot v + o(v)
\]
\[
B(u+iv) = B(u) + o(v),
\]
where little-o is taken as $(u+iv) \to 0$.

To show this note that if some
$f(x) \in \R\{x_1,\dots, x_d\}$ vanishes to order $2$ at $0$
then we have for small 
$u\in \R^d$, $v \in [0,\infty)^{d}$,
\[
f(u+iv) = f(u) + o(v).
\]
This follows from 
\[
f(u+iv) - f(u)= \int_{0}^{1} \nabla f(u+itv)\cdot iv dt
\]
using $|\nabla f(u+iv)| = O(|u+iv|)$.

Now consider
$\tilde{A}(u+iv) := A(u+iv) - \nabla \phi(0)\cdot(u+iv)$,
which vanishes to order $2$ and has real coefficients.
By the above
\[
\tilde{A}(u+iv) = \tilde{A}(u) + o(v) 
\]
and taking imaginary parts we have
\[
\Im (A(u+iv)) - \nabla \phi(0)\cdot v = o(v)
\]
since $\nabla \phi(0) \in \R^d$ and $\tilde{A}(u) \in \R$.
For the other estimate, note that $B$ vanishes to
order 2 already so the claim $B(u+iv) = B(u) + o(v)$
is immediate.

Therefore, since $\nabla \phi(0)$ has all positive entries and $v\in [0,\infty)^d$
\[
\Im (\phi(x)) \geq \nabla \phi(0)\cdot v + B(u) + o(v) \gtrsim |v| + B(u)
\gtrsim |v|  + |B(u+iv)|.
\]
\end{proof}

\begin{prop} \label{genprop}
Assume $\phi(x) \in \C\{x_1,\dots, x_d\}$, $\phi(0) = 0$, $\nabla \phi(0) \in (0,\infty)^{d}$, 
and $\Im \phi(x) \geq 0$ for $x \in \R^d$. 
Then, for $q(x,z) \in \C\{x_1,\dots, x_d, z\}$,
\[
\frac{q(x,z)}{z + \phi(x)}
\]
is bounded on $\overline{\uhp}^{d+1} \cap \D_{\epsilon}^{d+1}$
for some $\epsilon>0$
if and only if $q$ belongs to the ideal
\[
(z + (\Re \phi)(x), \mathrm{IC}((\Im \phi)(x)) ),
\]
    where
\[
\mathrm{IC}((\Im \phi)(x)) = \{ f(x) \in \C\{x_1,\dots, x_d\}: |f(x)| \lesssim (\Im \phi)(x)
\text{ for } x \in \R^d \text{ near } 0\}.
\]
In particular, 
$q(x,z)$ belongs to
$\mathcal{I}_p^{\infty}$ if and only if $q(x,-\Re \phi(x)) \in IC(B(x))$.
%we need only examine $q(x,z)$ on $\R^{d+1}$ to local
%boundedness (that is, boundedness on a neighborhood of $0$ intersected
%with $\uhp^{d+1}$).
\end{prop}

\begin{proof}
Again we write $\phi(x) = A(x) + i B(x)$.
By Proposition \ref{Im-est}, for $x = u+iv \in \uhp^{d}$
\[
|z+ \phi(x)| \geq \Im(z) + \Im (\phi(x)) \gtrsim \Im(z) + |B(x)| + |v|
\]
so we see that
\[
\frac{B(x)}{z+ \phi(x)}
\]
is locally bounded in $\uhp^{d+1}$. This implies
\[
\frac{z+ A(x)}{z+\phi(x)}
\]
has the same property.

Next we show that if $|f(u)| \lesssim B(u)$ for $u \in \R^d$ near $0$, 
then
$\frac{f}{z+\phi}$ is locally bounded in $\uhp^{d+1}$.
We have $f(u+iv) = f(u) + O(v)$
so that $|f(u+iv)| \lesssim B(u)  + |v| \lesssim \Im(\phi(x))$
by Proposition \ref{Im-est}. 
This implies $\frac{f}{z+\phi}$ is locally bounded in $\uhp^{d+1}$.
Therefore, 
everything in the ideal
\[
(z + A(x), \mathrm{IC}(B(x)) )
\]
yields a numerator for a locally bounded function.

Conversely, suppose $q \in \C\{x_1,\dots,x_d,z\}$
and $\frac{q}{z+\phi}$ is locally
bounded in $\uhp^{d+1}$.
We can write 
\[
q(x,z) = q_0(x) + (z+ A(x)) q_1(x,z)
\]
for $q_0(x) \in \C\{x_1,\dots, x_d\}$, $q_1(x,z) \in \C\{x_1,\dots, x_d,z\}$.
As $z+A(x)$ already belongs to the ideal, we need only show $q_0(x)$ belongs
to the ideal in question.
We have $|q(x,z)| \lesssim |z+\phi(x)|$ by assumption locally in $\uhp^{d+1}$
so that by continuity this extends to $\R^{d+1}$ near $0$.
Setting $z = -A(x)$ we have $|q_0(x)| \lesssim |B(x)|$ for small $x\in \R^d$
which means $q_0(x) \in \mathrm{IC}(B(x))$.

This last argument proves that a given $q(x,z)$ belongs to
$\mathcal{I}_p^{\infty}$ if and only if $q(x,-A(x)) \in IC(B(x))$.
\end{proof}

While our focus is on an isolated singularity at $0 \in \R^d$, Proposition \ref{genprop} does allow us to calculate the ideal of admissible numerators
in specific examples with non-isolated singularities.

\begin{example}\label{nonisoex}
Consider $2-xy-z$ which is non-vanishing on $\D^3$; this converts to the $\uhp^3$-stable polynomial 
\[
p(x,y,z) = x+y+z - 2i (xz + yz) - xyz.
\]
Note $p(x,-x,0) \equiv 0$.
The zero set through $(0,0,0)$ is parametrized by
$z + \phi(x,y) = 0$ for 
\[
\phi(x,y) = \frac{x+y}{1-2i(x+y) - xy}.
\]
Note that
\[
\Im \phi(x,y) = \frac{2(x+y)^2}{(1-xy)^2 + 4(x+y)^2}
\]
\[
\Re \phi(x,y) = \frac{(1-xy)(x+y)}{(1-xy)^2 + 4(x+y)^2}.
\]
Since $\Im \phi(x,y) \asymp (x+y)^2$ we can reduce 
$\Re \phi(x,y)$ mod $((x+y)^2)$
to obtain 
\[
\frac{x+y}{1-xy}.
\]
We conclude that the ideal of admissible numerators
is given by
\[
\left(z + \frac{x+y}{1-xy}, (x+y)^2\right) = (x+y+z - xyz, (x+y)^2)
\]
since $(1-xy)$ is a unit. Note that the minimal number of generators here 
is similar to the two varable case and is 
smaller compared to the ideals associated with $p$ having 
an isolated zero at $0$.$\quad \diamond$
\end{example}

\begin{proof}[Proof of Theorem \ref{thmImpos}]
By Remark \ref{phiremark}, $\nabla \phi(0)$ has all
positive entries when $\Im \phi_{2L}$ is positive definite.
Theorem \ref{thmImpos} now follows from Proposition \ref{genprop}
because $\Im \phi (x) \asymp |x|^{2L}$ on $\R^d$ and we can
reduce $\Re \phi(x)$ mod $(x)^{2L}$ to reduce to just the
Taylor polynomial of $\Re \phi(x)$ of degree less than $2L$.
\end{proof}

\begin{proof}[Proof of Proposition \ref{prop:sqpowers}]
Example \ref{faveex} shows that $L=1$ is possible. 

Now let $L\geq 2$ be an integer. Let $p$ be from Example \ref{faveex}.
Define the reflection polynomial 
\[\bar{p}(x_1, x_2, z)=x_1+x_2+z+2i(x_1x_2+x_1z+x_2z)-3x_1x_2z,\]
and set \[g(x_1,x_2,z)=i\frac{p(x_1,x_2,z)+\bar{p}(x_1,x_2,z)}{p(x_1,x_2,z)-\bar{p}(x_1,x_2,z)}.\] Then $g$ is analytic in $\mathbb{H}^3$, maps $\mathbb{H}^3$ to $\mathbb{H}$, and has $g(x_1,x_2,z)\in \mathbb{R}$ for $(x_1, x_2, z)\in \mathbb{R}^3$. In other words, $g$ is a real rational Pick function \cite[Section 6]{BPS2}.

For $x=(x_1,x_2)\in \overline{\mathbb{H}^2}$ fixed, set 
\[g_{x}(z)=g(x_1,x_2, z), \quad z\in \mathbb{H}.\]
For each $\xi\in \mathbb{R}^2$, the function $g_{\xi}(z)$ is either constant, or a rational self-map of the upper half-plane to itself, hence a M\"obius transformation. Letting $g^1=g$, we now set $g^2(x, z)=g(x, g_{x}(z))$, and similarly define $g^L=g^{L-1}(x, g_{x}(z))\colon \mathbb{H}^3\to \mathbb{H}$ for $L=3,4,\ldots$. Then each $g^L$ has the property that $z\mapsto g^L(\xi, z)$ maps $\mathbb{H}$ to itself for $\xi \in \mathbb{R}^2$, or else is constant.

Consider the function $\beta(w) =\frac{1+iw}{1-iw}$, which conformally maps $\mathbb{H}$ to $\mathbb{D}$ with $\beta(0)=1$. Let $f^L = \beta \circ g^L$ and $\phi^L  = \beta \circ g^L \circ \beta^{-1}$, where $\beta^{-1}$ is applied to each input of $g^L$ separately. One can easily check that each $\phi^L$ equals the rational inner function $\phi^N_d$ studied in Example 8 in \cite{S23} for $N=L$ and $d=3.$ Now the conclusions of that example can be applied directly. Specifically, one can write $\phi^L = \frac{q_L}{r_L}$, where the  polynomials $q_L, r_L$ have no common terms, vanish at $(1,1,1)$, and have degree $1$ in the third variable. The conclusions in Example 8 in \cite{S23} immediately imply that for $\theta_1, \theta_2 \in \mathbb{R}$ near $0$, 
\begin{equation} \label{eqn:generalL1} 
\left \{ 1-|z_3| : q_L(e^{i\theta_1}, e^{i\theta_2},z_3)=0 \right\} \asymp\left( \theta_1^2 + \theta_2^2\right) ^L. 
\end{equation}
To translate this to $f^L$, note that $f^L = \phi^L \circ \beta$. Then the discussion in \cite[Section 2.1]{BKPS} about changing domains from $\mathbb{D}^d$ to $\mathbb{H}^d$ implies that $f_L = q \frac{\bar{p}_L}{p_L},$ where $q$ is a two-variable unit near $(0,0)$ and $p_L$ is a stable polynomial in $\mathbb{H}^3$ with degree $1$ in $z$. The properties of $\beta$ allow one to translate \eqref{eqn:generalL1} to the following statement about $p_L$ for $x_1, x_2 \in \mathbb{R} $ near $0$:
\begin{equation} \label{eqn:generalL2} 
\left \{ |\Im(z)| \colon p_L(x_1,x_2, z)=0\right \}
=\left \{ \Im(z) \colon \bar{p}_L(x_1,x_2, z)=0 \right \}  \asymp \left( x_1^2 + x_2^2\right )^L.
\end{equation} 
(While we omit the change-of-variables computation from \eqref{eqn:generalL1}  to \eqref{eqn:generalL2} here, the interested reader could see \cite{BPS1}; the details of a very similar  conversion are given in Step 2 of the proof of Theorem 3.3.) Then \eqref{eqn:generalL2}  implies that $p_L$ has an 
associated $\phi_{2L}\in \mathbb{C}[x_1,x_2]$ with $(\Im\phi_{2L})(x_1,x_2)\asymp \left( x_1^2 + x_2^2\right )^L,$ which completes the proof. 
\end{proof}

\begin{example}
Applying the construction from the proof of Proposition \ref{prop:sqpowers} with $L=2$ to the polynomial in Example \ref{faveex} produces 
\begin{multline}p_2(x,y,z)=x +y+ 2 i((x+y)^2-2x^2y^2) -2 (x^2 y +x y^2)  \\
+ (1 + 2 i (x+y-2x^2y-2xy^2) - 2 (x  + y)^2) z.
    \end{multline}
 For $p_2$, we compute that
\[
\phi(x_1,x_2)=x+y+2(x^3+2x^2y+2xy^2+y^3)+4i(x^4+2x^3y+3x^2y^2+2xy^3+y^4)+\textrm{higher order}.
\] 
For this example, one can actually check directly (viz. \cite[p.1155]{BPS3d}) that  $\Im \phi_4(x,y)\asymp (x^2+y^2)^2$, as guaranteed by Proposition \ref{prop:sqpowers}. Setting
\[H(x,y)=x+y+2(x^3+2x^2y+2xy^2+y^3),\]
we thus have
\[
\mathcal{I}^{\infty}_{p_2}=(z+H(x,y), (x,y)^4). \qquad\diamond
\]
\end{example}

Now we consider Proposition \ref{prop:degex}. 
In what follows, $K$ is the contact order of the two-variable 
polynomial $q$ at $(0,0)$.
This means that $K$ a positive even integer that measures 
how the zero set of $q$ approaches $(0,0)$ in the following sense: 
\[ \text{inf} \left\{ |\Im (y)|: q(x,y) = 0\right\}  \asymp |x|^K,\]
for $x \in \mathbb{R}$ sufficiently close to $0$. 
See ``Theorem (Puiseux Factorizations)'' in the 
introduction of \cite{BKPS} and the later proof of that result for more information about contact order in the case of a pure stable polynomial
with a single, irreducible, degree $1$ Weierstrass polynomial.

\begin{proof}[Proof of Proposition \ref{prop:degex}] 
Since $\frac{\partial q}{\partial y}(0) \ne 0$, we can factor
\[
q(x,y) = u(x,y)\left(y + \psi(x)\right)
\]
for $\psi(x) \in\C\{x\}$ with $\psi(0)=0$, 
and $u(x,y) \in \C\{x,y\}$ with $u(0,0)\ne 0$.
Writing $\psi(x) = \sum_{k=1}^{\infty} a_k x^k$
we can apply Proposition \ref{phithm}
to see $a_1>0$ and there exists an even natural number
$\tilde{K}$ such that $a_j \in \R$ for $j<\tilde{K}$
and $\Im a_{\tilde{K}} >0$.
Thus for $x \in \R$ and $q(x,y) = 0$, 
\[
|\Im y| = |\Im(a_{\tilde{K}}) x^{\tilde{K}} + O(x^{\tilde{K}+1})|
\asymp |x|^{\tilde{K}}.
\]
By definition of contact order we must have $K = \tilde{K}$.
%the definition of the contact order $K$ of $q$ at $(0,0)$ 
%implies that for $x$ real and near $0$,
%\[ |x|^K \asymp |\Im (\psi(x))| = \left | \sum_{k=1}^{\infty} \Im(a_k) x^k \right|. \]
%Thus, the contact order $K$ must equal the smallest $ k \in \mathbb{N}$ such that $ \Im (a_k) \ne 0.$  Furthermore, $a_1 \ne 0$ because $q$ has only finitely many zeros on $\mathbb{R}^2$ and $a_1$ is positive  because if it were negative, then $q$ would have zeros in $\mathbb{H}^2$.

Define $p$ as in the statement of Proposition \ref{prop:degex}. %\color{black}
Then, by the properties of $\beta = i(1-z)/(1+z)$ and $\beta^{-1}$, 
$p$ is a stable polynomial on $\mathbb{H}^3$ 
and has finitely many zeros on $\mathbb{R}^3$.
Furthermore, its zero set $\mathcal{Z}_p$ near $0$ is parameterized by $z+\phi(x,y)=0$, where
\[ \phi(x,y) = \psi\left( \frac{i(x+y) +2xy}{2i+x+y}\right). \]
One can check that 
\[ \frac{\partial^2 \phi}{\partial x^2} (0,0) = \frac{1}{4} \psi''(0) + \frac{i}{2} \psi'(0).\]
Since $K>2$, we know $\Im \psi''(0)=2\Im a_2=0$. Thus, using the notation of Proposition \ref{phithm},  $\Im \phi_2(x,y)$ includes the term $\frac{a_1}{4} x^2$ and so $\Im \phi_2(x,y) \not \equiv 0.$ However, restricting to $y=x$ gives
\[  \phi(x,x) = \psi(x)\]
and so $\Im \phi_k(x,x)=0$ for $k<K.$
\end{proof}
\color{black}

\section{Proofs part II: General isolated points}\label{proofspII}

First, we prove Theorem \ref{thm:isoideal}
assuming Lemma \ref{lociso}.

\begin{proof}[Proof of Theorem \ref{thm:isoideal}]
The zeros of $p$ near $0$ in $\R^{d+1}$ are given by
$z+ \phi(x) = 0$.
Since we assume the zero at $0$ is isolated
with respect to $\R^{d+1}$, by Remark \ref{phiremark}
we see that $\Im \phi(x) >0$ for $x$
near by not equal to $0$.
Applying Lemma \ref{lociso} to $f(x) = \Im \phi(x)$,
there exists a polynomial $g(x) \asymp f(x)$
and a natural number $K$ such that $f(x) \gtrsim |x|^K$.

Since $p$ has an isolated zero at $0$
in $\R^{d+1}$ we necessarily have $\nabla \phi(0)$
with all positive entries by Proposition \ref{phithm}.
We can then apply Proposition \ref{genprop}
to see
\[
\mathcal{I}^{\infty}_p = (z+\Re \phi(x), \mathrm{IC}( \Im \phi(x))).
\]
We can replace $\Im \phi(x) = f(x)$
with $g(x)$.
Since $g(x) \gtrsim |x|^K$, the ideal
$(x)^K$ is contained in the admissible numerator
ideal and we can reduce $\Re \phi(x)$ mod $(x)^K$
without changing the ideal.
Now, $\Re \phi(x)$ is equivalent to its $K-1$-th
order Taylor polynomial. Call this polynomial $H(x,y)$.

We arrive at the following representation of the 
admissible numerator ideal
\[
\mathcal{I}^{\infty}_p = (z+ H(x), \mathrm{IC}(g(x))).
\]
\end{proof}

As mentioned in Remark \ref{kollarremark}, Lemma \ref{lociso}
follows from the Łojasiewicz inequality.  We shall give a brief explanation
for this as well as a more explicit proof of Lemma \ref{lociso} in the case $d=2$.
Recall that the Łojasiewicz inequality says that given a real
analytic function $F:U\to \R$ defined on an open subset $U$ of $\R^d$, 
and given compact $V\subset U$, there exist positive constants $C, K$
such that for $x\in V$
\[
\text{dist}(Z_F, x)^{K} \leq C|F(x)|
\]
where $Z_F$ is the zero set of $F$.

\begin{proof}[Proof of Lemma \ref{lociso} for general $d$]
We apply the Łojasiewicz inequality to $f$ restricted to a 
neighborhood $U$ of $0$ where by assumption $Z_f\cap U = \{0\}$.
Then, in a compact neighborhood $V$ of $0$ within $U$ there exists 
$K$  such that
\[
|x|^K \lesssim f(x).
\]
We can take $K$ to be an integer since the inequality is true for larger
values.  Now let $g$ be the $K$-th order Taylor polynomial of $f$.
Then, since $f(x)-g(x) = O(|x|^{K+1})$, we have 
$c_1|x|^K\leq f(x) \leq g(x) + c_2|x|^{K+1}$
so that
$g(x) \gtrsim |x|^K$ near $0$.
Then, we automatically have $f(x) \lesssim g(x)$
and similarly $g(x)\lesssim f(x)$.    
\end{proof}

\begin{proof}[Proof of Lemma \ref{lociso} for $d=2$]
First, by the Weierstrass preparation theorem we
can write
\[
f(x,y) = u(x,y) \prod_{j=1}^{k} f_j(x,y)
\]
where $u \in \R\{x,y\}$ is a unit,
and each $f_j \in \C\{x\}[y]$ is an irreducible
Weierstrass polynomial with no zeros in $\R^2\setminus \{(0,0)\}$.
Since $f$ has real coefficients, 
each $f_j$ either has real coefficients
or comes with a conjugate pair $\bar{f}_j \in \{f_1,\dots, f_k\}$.

\begin{claim} \label{gnj}
For each $j \in \{1,\dots, k\}$ there exists
a Weierstrass polynomial
$g_{N,j}(x,y) \in \C[x,y]$ (which we emphasize has \emph{polynomial}
coefficients) 
whose coefficient polynomials agree with those of $f_j(x,y)$
to arbitrarily high order for $N$ large enough and which 
satisfies an estimate
$|g_{N,j}(x,y)| \gtrsim |(x,y)|^{K_j}$.
Also, if $f_j$ has a conjugate $f_i = \bar{f}_j$ then
$g_{N,i} = \bar{g}_{N,j}$.
\end{claim}

Assuming the claim temporarily, we define
\[
g_N(x,y) := \prod_{j=1}^{k} g_{N,j}(x,y).
\]
Note that $g_N$ has real coefficients 
because the Puiseux branches of $f$ are either real or
occur in conjugate pairs.
Also, $g_N$
satisfies
\[
|g_N(x,y)| \gtrsim |(x,y)|^{K}
\]
for $K = \sum_{j=1}^{k} K_j$
and agrees with 
\[
W(x,y) := \prod_{j=1}^{k} f_j(x,y)
\]
up to arbitrarily high order as we increase $N$.
Since both $W(x,y)$ and $g_N(x,y)$ are monic in $y$ 
we can 
choose $N$ so that
\[
|g_N(x,y) - W(x,y)| \lesssim |x|^{K+1}
\]
where we absorb contributions from $y$ into the implicit constant.

Then,
\[
\left| 1- \frac{W(x,y)}{g_N(x,y)}\right| 
\lesssim \frac{|x|^{K+1}}{|(x,y)|^{K}} <1/2
\]
for $x$ sufficiently small, so
\[
W(x,y) \asymp g_N(x,y).
\]

This proves 
\[
f(x,y) \asymp g_N(x,y)
\]
and we have an appropriate bound below on both
$f(x,y)$ and $g_N(x,y)$.

Now we prove Claim \ref{gnj}.
We fix $j\in \{1,\dots, k\}$ and factor $f_j$ into 
\[
\prod_{n=1}^{r} (y - \psi(\mu^{n} x^{1/r}))
\]
where $r\in \mathbb{N}$, $\mu = \exp(2\pi i/r)$, and $\psi\in \C\{x\}$ by
the Newton-Puiseux theorem (see \cite[Chapters 6-7]{Fish} for details).
For $f_j(x,y)$ to be non-vanishing
for $(x,y)\ne (0,0)$ we must have
that $\psi(\mu^n x^{1/r}) \notin \R$ 
whenever $x\ne 0$.
We can discuss this with the fixed
choice of branches---i.e. $x^{1/r} > 0$ 
for $x>0$ and $x^{1/r} = |x|^{1/r} \exp(i\pi/r)$
for $x<0$---since as we vary over $n$ we cover all 
of the branches.

Writing out the power series 
$\psi(t) = \sum_{m\geq 1} \psi_m t^m$
for $x>0$, we have
\[
\psi(\mu^n x^{1/K}) 
= \sum_{m \geq 1} \psi_m \mu^{nm} x^{m/K},
\]
and there must exist a first coefficient, 
say $m=M^{+}_n$,
such that $\psi_m \mu_j^{nm} \notin \R$,
and for $x<0$ 
\[
\psi(\mu^n x^{1/r}) 
= \sum_{m \geq 0} \psi_m \mu^{nm} \exp(i\pi m/r) |x|^{m/r}
\]
there must exist a first
coefficient, say $m=M^{-}_n$, such that
$\psi_m \mu^{nm} \exp(i\pi m/r) \notin \R$.
In particular, for $M_{n} = \max\{M^{+}_n, M^{-}_n\}$,
\[
|\Im \psi(\mu^n x^{1/r})| \gtrsim |x|^{M_{n}/r}.
\]
Also, (since $|A|^N \lesssim_N |A-B|^N + |B|^N$)
\[
|y- \Re \psi(\mu^n x^{1/r})|\gtrsim 
|y - \Re \psi(\mu^n x^{1/r})|^{M_{n}+1} \gtrsim |y|^{M_{n}+1} - |x|^{(M_{n}+1)/r}, 
\]
for $x,y$ small enough.
Therefore,
\begin{equation} \label{psiest}
|y- \psi(\mu^n x^{1/r})| \gtrsim |y|^{M_n+1} + |x|^{M_{n}/r}
\gtrsim |(x,y)|^{M_{n}+1}.
\end{equation}

Let $M = \max\{M_{n}: n=1,\dots, r\}$ and $N >M$.
Define
\[
\psi^{[N]} (t) = \sum_{m=1}^{N} \psi_m t^m \in \C[t]
\]
and
\[
g_{N,j}(x,y) = \prod_{n=1}^{r}(y - \psi^{[N]}(\mu^n x^{1/r})).
\]
Now, $g_{N,j}$ is necessarily a polynomial
since we have an expression symmetric over all branches of $x^{1/r}$.
For $N>M$, the estimate \eqref{psiest} holds for
$\psi^{[N]}$ in place of $\psi$ 
since $N>M$ implies that we capture all of the coefficients
that contribute to the estimate \eqref{psiest}.

Therefore by \eqref{psiest} applied to $\psi^{[N]}$,
$|g_{N,j}(x,y)| \gtrsim |(x,y)|^{K_j}$
for $K_j := \sum_{n=1}^{r} (M_n+1)$.
As we increase $N$, the coefficients
of $g_{N,j}(x,\cdot)$ match those of $f_j(x,\cdot)$ 
to arbitrarily high order in $x$.
Also, note that
our construction respects conjugate pairs.
This proves Claim \ref{gnj} as well as the lemma.
\end{proof}

\section*{Acknowledgements}
We thank J\'anos Koll\'ar for helpful correspondence concerning higher-dimensional aspects of the work in \cite{Kollar} and especially for Remark \ref{kollarremark}.
Thank you to the referees for a careful reading and numerous useful suggestions.

\end{document}